\documentclass[12pt]{amsart}
\usepackage{amssymb,latexsym,amsthm}
\usepackage{amsmath,amsfonts,amssymb}
\usepackage{graphicx}
\usepackage{color}
\usepackage{mathrsfs}

\newcommand{\es}{\mathbb{B}(E^{*})}
\newcommand{\wx}{\widetilde{X}}
\newcommand{\Tb}{T_0}
\newcommand{\Sf}{S_{fx}}

\newcommand{\lx}{\operatorname{Lip}(X)}

\newcommand{\lxe}{\operatorname{Lip}(X, E)}
\newcommand{\Sig}{\Sigma^{n}_{i=1}}
\newcommand{\lxo}{\operatorname{Lip}(X_1, E_1)}
\newcommand{\lxt}{\operatorname{Lip}(X_2, E_2)}

\newcommand{\ma}{\mathcal{A}}
\newcommand{\hm}{H(\mathcal{A})}
\newcommand{\hmo}{H(\mathcal{A}_1)}
\newcommand{\hmt}{H(\mathcal{A}_2)}

\newcommand{\lxm}{\Lip(X,\mathcal{A})}

\newcommand{\lxom}{\Lip(X_1,\mathcal{A}_1)}
\newcommand{\lxtm}{\Lip(X_2,\mathcal{A}_2)}
\newcommand{\Lip}{\operatorname{Lip}}

\newcommand{\et}{e_{\theta}}

\newtheorem{theorem}{Theorem}[section]
\newtheorem{lemma}[theorem]{Lemma}

\newtheorem{prop}[theorem]{Proposition}

\theoremstyle{definition}
\newtheorem{definition}[theorem]{Definition}

\theoremstyle{remark}

\begin{document}

\author{
Shiho~Oi
}
\address{
Department of Mathematics, Faculty of Science, Niigata University, Niigata
950-2181 Japan.
}
\email{shiho-oi@math.sc.niigata-u.ac.jp
}

\title[]
{Isometries and hermitian operators on spaces of vector-valued Lipschitz maps
}




\begin{abstract} 
We study hermitian operators and isometries on spaces of vector-valued Lipschitz maps with the sum norm: $\|\cdot\|_{\infty}+L(\cdot)$. There are two main theorems in this paper. Firstly, we prove that every hermitian operator on $\Lip(X,E)$, where $E$ is a complex Banach space, is a generalized composition operator. Secondly, we give  a complete description of  unital surjective complex linear isometries on $\Lip(X,\ma)$ where $\ma$ is a unital factor $C^{*}$-algebra. These results improve previous results stated by the author. 

\end{abstract}
\maketitle
\section{Introduction and Main results}

Given a compact metric space $X$ and a complex Banach space $(E,\|\cdot\|_{E})$, 
a map $F: X \to E$ is said to be Lipschitz if 
\[
L(F):=\sup_{x \neq y \in X}\left\{ \frac{\|F(x)-F(y)\|_{E}}{d(x,y)} \right\} < \infty.
\]
We denote a space of all E-valued Lipschitz maps on $X$ by $\lxe$.  In the case $E=\mathbb{C}$, we simply write $\lx$. The Lipschitz space $\lxe$ is a Banach space with the sum norm; 
\[
\|F\|_{L}=\sup_{x \in X}\|F(x)\|_{E}+L(F),  \quad F \in \lxe.
\]
In particular, $\lxe$ endowed with $\|\cdot\|_{L}$ is a Banach algebra if $E$ is a Banach algebra. 

\subsection{Surjective linear isometries}
Let $\ma$ be a unital $C^{*}$-algebra. We study unital surjective linear isometries on $\lxm$ with $\|\cdot\|_{L}$. We explain the motivation for our study. 
Kadison in \cite{kadison} obtained the following characterization of surjective complex linear isometries between unital $C^{*}$-algebras. 
Let $\ma_i$ be unital $C^{*}$-algebras for $i=1,2$. Let $U:\ma_1 \to \ma_2$ be a surjective linear isometry. Then there are a unitary element $u \in \ma_2$ and  a Jordan $*$-isomorphism $\psi: \ma_1 \to \ma_2$ such that 
$U(a)=u\psi(a)$ for any $ a \in \ma.$
This has a remarkable and beautiful consequence that the unital surjective linear isometries between unital $C^{*}$-algebras are Jordan $*$-isomorphisms.
A lot of  researchers  have been attracted to consider whether every surjective linear isometry on algebras is closely related to an isomorphism on the algebras. We deal with surjective linear isometries on Banach algebras of continuous maps taking values in a unital $C^{*}$-algebra. For any unital $C^{*}$-algebra $\ma$, we denote by $C(K,\ma)$ the Banach algebra,  with the supremum norm, of  all continuous maps on a compact Hausdorff space $K$ taking values in $\ma$. Let us consider surjective linear isometries between $C(K,\ma)$. 
Since $C(K,\ma)$ is a unital $C^{*}$-algebra, the celebrated theorem due to Kadison tells that every unital surjective linear isometry  is  a Jordan $*$-isomorphism. In particular, if $\ma_i$ are unital factor $C^{*}$-algebras for $i=1,2$, in \cite[Corollary 5]{kho} they showed that every surjective linear isometry $U: C(K_1,\ma_1) \to C(K_2,\ma_2)$ is a weighted composition operator of the form 
\begin{equation}\label{eeee}
UF(y)=u\psi_y(F(\varphi(y))),
\end{equation}
where $\varphi:K_2 \to K_1$ is a homeomorphism, $\{\psi_{y}\}_{y \in K_2}$ is a strongly continuous family of Jordan $*$-isomorphisms from $\ma_1$ onto $\ma_2$, and $u \in C(K_2,\ma_2)$ is a unitary element.  
One may wonder whether any surjective linear isometries from $\lxom$ onto $\lxtm$ is also a weighted composition operator similar to (\ref{eeee}). 
First we introduce the results by the author in \cite{hermitian1}. We showed every hermitian operator on $\lxe$ is a generalized composition operator under the more restrictive condition that $E$ is of finite dimension. 
Furthermore, we obtained the following theorem by using the notion of hermitian operators.  We denote the Banach algebra of complex matrices of order $n$ by $M_{n}(\mathbb{C})$.  
\begin{theorem}[Theorem 3.3 in \cite{hermitian1}]
Let $X_i$ be compact metric spaces for $i=1,2$. The map $U : (\Lip(X_1,M_{n}(\mathbb{C})), \|\cdot\|_{L}) \to (\Lip(X_2,M_{n}(\mathbb{C})), \|\cdot\|_{L})$ is a linear surjective isometry such that $U(1) = 1$ if and only if there exist
a unitary matrix $ V \in M_{n}(\mathbb{C})$ and a surjective isometry $\varphi:X_2 \to X_1$  such that
\[
(UF)(x)=VF(\varphi(x))V^{-1}, \quad F \in \Lip(X_1,M_{n}(\mathbb{C})), x \in X_2
\]
or
\[
(UF)(x)=VF^{t}(\varphi(x))V^{-1}, \quad F \in \Lip(X_1,M_{n}(\mathbb{C})), x \in X_2,
\]
where $F^{t}(y)$ denote transpose of $F(y)$ for $y \in X_1$.
\end{theorem}
 Although the arguments of the proof have remained limited to the case that $E$ is finite dimensional, this is the first result on surjective linear isometries on $\lxe$, where $E$ is a non-commutative Banach algebra. In this framework, it seems natural to ask questions about further developments. 
The aim of this paper is to develop our knowledge on hermitian operators and isometries on $\lxe$ and establish a infinite dimensional version of \cite{hermitian1}. 
More precisely, we prove the next theorem.
\begin{theorem}\label{isometrydd}
Let $X_i$ be compact metric spaces and $\ma_i$ unital factor $C^{*}$-algebras for $i=1,2$. The map $U: (\lxom, \|\cdot\|_{L}) \to (\lxtm, \|\cdot\|_{L})$ is a surjective complex linear isometry such that $U(1)=1$ if and only if there exist a unital surjective complex linear isometry $\psi:\mathcal{A}_1 \to \mathcal{A}_2$ and a surjective isometry $\varphi:X_2 \to X_1$ such that
\[
UF(y)=\psi(F(\varphi(y))), \quad F \in \lxom, y \in X_2.
\] 
\end{theorem}
Indeed, Theorem \ref{isometrydd} also gives an answer to the above question (that is  whether any surjective linear isometries from $\lxom$ onto $\lxtm$ is also a weighted composition operator similar to (\ref{eeee})). 

In case that $E$ is a finite dimensional Banach space,  it follows from \cite[Lemma 2.1]{hermitian1}  that  $\Lip(X) \otimes E=\lxe$. If $E$ is of infinite-dimension, $\Lip(X) \otimes E$ does not coincide with $\Lip(X,E)$.  Moreover it is not known whether $\Lip(X) \otimes E$ is dense in $\lxe$ with $\|\cdot\|_{L}$ or not. When $E$ is of infinite dimension, the representation of $\lxe$ is more complicated. Thus to describe isometries and hermitian operators on $\lxe$, where $E$ is of infinite dimension,  is much more difficult than the case that $E$ is of  finite dimension. 
In order to achieve any further progress, we need to make several improvements and extensions compared to the paper \cite{hermitian1}. 

The paper is organized as follows.
In the rest of the introduction we provide basic background on the study of hermitian operators. In Section \ref{hhhhh} we study hermitian operators on $\lxe$. The main theorem of Section \ref{hhhhh} is Theorem \ref{hermitian}.  For any a complex Banach space $E$, we prove that every hermitian operator on $\lxe$ is a generalized composition operator. This is a generalization of the characterization of hermitian operators on $\lxe$ for finite dimensional Banach spaces $E$  in \cite{hermitian1}.  In Section \ref{33333}, we introduce the concept of T-sets due to Myers. By the notion of T-sets we present properties of the unit ball of the dual space of $\lxe$.  Indeed the extreme points of the unit ball of the dual space of $\lxe$ are quite complicated. Thus we study T-sets instead of  extreme points. The main statement in Section \ref{33333} is Proposition \ref{3.6}. We show that if  a surjective linear isometry between $\lxe$ is a weighted composition operator when restricted to $\Lip(X) \otimes E$, then it is a weighted composition operator. Since the representation of $\Lip(X) \otimes E$ is much easier than that of $\lxe$, Proposition \ref{3.6} is successful in describing the surjective linear isometries on $\lxe$. In Section \ref{44444} we present the proof of Theorem \ref{isometrydd}. 

\subsection{Hermitian operators}
A bounded operator $T$ on a complex normed space $(V,\|\cdot\|_{V})$ is hermitian if
$[Tv,v]_{V} \in \mathbb{R}$ for any $v \in V$, where $[\cdot,\cdot]_{V}$ is a semi-inner product on $V$ that is compatible with the norm $\|\cdot\|_{V}$. The definition does not depend on the choice of semi-inner products (see \cite{BD}).
A complete description of hermitian operators on Banach spaces has been studied for a long period of time.  We refer the reader to  \cite{FJB, FJB08} for further information about hermitian operators.

Fleming and Jamison in \cite{FJ} turned their attention to the vector-valued case. Let $E$ be a complex Banach space. They obtained the first characterization for  hermitian operators between Banach spaces of $E$-valued continuous functions as follows:  

Let $T$ be a  hermitian operator on $C(K,E)$, where $K$ is a compact Hausdorff space. Then for each $t \in K$ there is a hermitian operator $\phi(t)$ on $E$ such that 
\[
TF(t)=\phi(t)(F(t)), \quad t \in K.
\]

What is the general form of hermitian operators between Banach spaces of $E$-valued Lipschitz maps?  
Botelho, Jamison, Jim\'enez-Vargas and Villegas-Vallecillos  in \cite{BJJMm} obtained a characterization for hermitian operators on $\lxe$ with the max norm; $\|\cdot\|_{M}=\max \{ \|\cdot\|_{\infty}, L(\cdot) \}$ as follows:

Let $X$ be a compact and 2-connected metric space and $E$ a complex Banach space. Then  $T: (\lxe,\|\cdot\|_{M}), \to (\lxe,\|\cdot\|_{M})$ is a hermitian operator
 if and only if there exists a hermitian operator $\phi: E \to E$ such that
\begin{equation*}
TF(x)=\phi(F(x)),  \quad F \in \lxe, \quad  x \in X.
\end{equation*}
How about the case $\lxe$ with $\|\cdot\|_{L}$?
One may think that each feature of the two norms does not make a big difference, but it is not. The studies of the classes of operators on $\lxe$ depend heavily on the properties of the norm. 
The standard approach to the studies of isometries or related operators on Banach spaces relies on a characterization of the extreme points of the closed unit ball of the corresponding dual spaces. But the extreme points of the closed unit ball of the dual space of  $(\lxe,\|\cdot\|_{L})$ are completely different from those of $(\lxe,\|\cdot\|_{M})$. The former is much complicated. For operators on $(\lxe,\|\cdot\|_{L})$, it is non trivial to derive a representation from the action of their adjoints,  
so we have to work quite hard to give a representation. 
Actually, in the case of hermitian operators on $(\lxe,\|\cdot\|_{L})$, difficulties to give a representation 
remain even if we have a representation of  hermitian operators on  $(\lxe,\|\cdot\|_{M})$. 
Indeed, Botelho, Jamison, Jim\'enez-Vargas and Villegas-Vallecillos proved hermitian operators between $\lx$ with $\|\cdot\|_{L}$ in \cite{BJJMs} are composition operators. Recently, the author of this paper generalized to $\lxe$, where $E$ is a finite dimensional complex Banach space in \cite{hermitian1}. But it has not been solved in general. In this paper, we give a complete representation for any complex Banach space $E$. 

\subsection{Notations and Remarks}
Throughout this paper, $X$,  $X_1$and $X_2$ are compact metric spaces, and  $E$, $E_1$ and $E_2$ are complex Banach spaces. In addition  $\ma$, $\ma_1$ and $\ma_2$ are unital $C^{*}$-algebras. For a unital $C^{*}$-algebra $\ma$, if its center is trivial, i.e. $A \cap A'=\mathbb{C}1$ we call it a unital factor $C^{*}$-algebra. 
 For Banach space $E$, we denote the closed unit ball of $E$ by $\mathbb{B}(E)$, and the closed unit ball of the dual space $E^{*}$ by $\es$.  We also denote the unit sphere of $E$ by $\mathbb{S}(E)$. 
For any $f \in \lx$ and $e \in E$, we define $f \otimes e: X \to E$ by
\[
(f\otimes e)(x)=f(x)e.
\]
We have $f \otimes e \in \lxe$ such that 
$\|f \otimes e\|_{\infty}=\|f\|_{\infty}\|e\|_{E}$
and
$L(f \otimes e )=L(f)\|e\|_{E}$. 
This implies that $\|f\otimes e\|_{L}=\|f\|_{L}\|e\|_{E}$. We see that $f \otimes e$ is an element of the algebraic tensor product space $\lx \otimes E$ with the crossnorm. 

Recall that the purpose of this paper is to generalize the theorems in \cite{hermitian1}. Although we need new approaches and additional arguments, some arguments remain valid. 
Similar arguments may be found in \cite{hermitian1}, but we adapt these to our setting and give proofs as accurately as possible. 
\section{A characterization of hermitian operators on $\lxe$}\label{hhhhh}
Firstly, we would like to consider hermitian operators on $(\lxe, \|\cdot\|_{L})$.
We write $\wx=\{(x,y) \in X^2 ; x \neq y\}$. 
Let $\beta(\wx \times \es)$ be the Stone-\v Cech compactification of  $\wx \times \es$.
For any $F \in \lxe$, we denote by $\widetilde{F}:\beta(\wx \times \es)\to \mathbb{C}$ the unique continuous extension of the bounded continuous function  $((x,y), e^{*})\mapsto e^{*}(\frac{F(x)-F(y)}{d(x,y)})$
on $\wx \times \es$. Since we have $\|\widetilde{F}\|_{\infty}=L(F)$  for any $F \in \lxe$, we can define a linear isometric embedding $\Gamma: (\lxe, \|\cdot\|_{L}) \to (C(X\times \beta(\wx \times \es) \times \mathbb{B}(E), E), \|\cdot\|_{\infty})$ by $\Gamma(F)(x, \xi, e)=F(x)+\widetilde{F}(\xi) e$. Moreover we define a set $P_{G}$ by
\[
P_{G}=\{t \in X\times \beta(\wx \times \es) \times \mathbb{B}(E); \|\Gamma(G)(t)\|_{E}=\|\Gamma(G)\|_{\infty}=\|G\|_{L}\}.
\]
\begin{lemma}\label{empty}
For any $G \in \lxe$, we have $P_{G} \neq \emptyset$.
\end{lemma} 
\begin{proof}
If $G=0$, we have $(x_0, \xi, e) \in P_{G}$  for any $(x_0, \xi, e) \in X \times \beta(\wx \times \es) \times \mathbb{B}(E)$. Thus let $G \in \lxe$ with $G \neq 0$. Since $\beta(\wx \times \es)$ is compact, there exists $\xi \in \beta(\wx \times \es)$ such that $|\widetilde{G}(\xi)|=\|\widetilde{G}\|_{\infty}=L(G)$. There are $x_0 \in X$ such that $\|G(x_0)\|_{E}=\|G\|_{\infty}$ and $\alpha \in \mathbb{C}$ with $|\alpha|=1$ such that $\alpha\widetilde{G}(\xi)=\|\widetilde{G}\|_{\infty}=L(G)$. We get
\begin{equation*}
\begin{split}
\|\Gamma(G)(x_0, \xi, \frac{\alpha}{\|G(x_0)\|_{E}}G(x_0))\|_{E}=\|G(x_0)+\widetilde{G}(\xi) \frac{\alpha}{\|G(x_0)\|_{E}}G(x_0)\|_{E}\\
=(1+L(G)\frac{1}{\|G\|_{\infty}})\|G\|_{\infty}=\|G\|_{\infty}+L(G)=\|G\|_{L}.
\end{split}
\end{equation*}
This implies that $(x_0, \xi, \frac{\alpha}{\|G(x_0)\|_{E}}G(x_0) )\in P_{G}$.
\end{proof}
By Lemma \ref{empty} and the axiom of choice, there exists a choice function 
\[
\Psi: \lxe \to X\times \beta(\wx \times \es) \times \mathbb{B}(E)
\]
such that $\Psi(G) \in P_{G}$ for every $G\in \lxe$. 
Let $[\cdot,\cdot]_{E}$ on $E$ be a semi-inner product  which is compatible with the norm of $E$. Define a map $[\cdot,\cdot]_{\Psi L}: \lxe \times \lxe \to \mathbb{C}$ given by
\begin{equation}\label{choice}
[F,G]_{\Psi L}=[\Gamma(F)(\Psi(G)), \Gamma(G)(\Psi(G))]_{E},\quad F, G \in \lxe.
\end{equation}
It is easy to check that  $[\cdot,\cdot]_{\Psi L}$ is a semi-inner product on $\lxe$ compatible with the norm $\|\cdot\|_{L}$.
Now we get the following lemma. The basic idea of the proof is the same as \cite[Lemma 2.3]{hermitian1}. 
\begin{lemma}\label{tako}
Let $T$ be a hermitian operator on $(\lxe, \|\cdot\|_{L})$. Then 
\[
T(1 \otimes e) \in 1 \otimes E
\]
for any $e \in E$.
\end{lemma}
\begin{proof}
Let $e\in E$. 
If $e=0$, then $T(1\otimes e)=T(0)=0=1\otimes 0 \in 1\otimes E$. 
Thus we assume that $0 \neq e \in \mathbb{B}(E)$.
Fix $x' \in X$, $(x,y)\in \wx$ and $e^{*}\in \es$. Let $\theta \in [0, 2\pi)$, we obtain
\begin{multline}\label{0.4}
\lefteqn{\Gamma(1 \otimes e)(x',((x,y),e^{i\theta}e^{*}),e)}\\
=(1 \otimes e)(x')+e^{i\theta}e^{*}\left(\frac{(1 \otimes e)(x)- (1\otimes e)(y)}{d(x,y)}\right) e=e+0 e =e.
\end{multline}
This implies that 
\[
\|\Gamma(1 \otimes e)(x',((x,y),e^{i\theta}e^{*}),e)\|_{E}=\|1\otimes e\|_{L}.
\]
Thus we get $(x',((x,y),e^{i\theta}e^{*}),e) \in P_{1\otimes e}$. Choose a choice function $\Psi_{\theta}: \lxe \to X\times \beta(\wx \times \es) \times \mathbb{B}(E)$ such that 
\[
\Psi_{\theta}(1 \otimes e)=(x',((x,y),e^{i\theta}e^{*}),e)
\]
and define a semi-inner product $[\cdot,\cdot]_{\Psi_{\theta}L}$ on $\lxe$ in the manner as in \eqref{choice}. Since $T$ is a hermitian operator, we have
$[T(1 \otimes e), 1\otimes e]_{\Psi_{\theta}L} \in \mathbb{R}$.
By (\ref{0.4}), it follows that 
\begin{equation}\label{0.5}
\begin{split}
\mathbb{R}&\ni [T(1 \otimes e), 1\otimes e]_{\Psi_{\theta}L}\\
&=[\Gamma(T(1 \otimes e))(\Psi_{\theta}(1 \otimes e)), \Gamma(1 \otimes e)(\Psi_{\theta}(1 \otimes e))]_{E}\\
&=[T(1 \otimes e)(x')+e^{i\theta}e^{*}\left(\frac{T(1 \otimes e)(x)-T(1 \otimes e)(y)}{d(x,y)}\right) e, e]_{E}\\
&=[T(1 \otimes e)(x'),e]_{E}+e^{i\theta}e^{*}\left(\frac{T(1 \otimes e)(x)-T(1 \otimes e)(y)}{d(x,y)}\right)\|e\|^2_{E}.
\end{split}
\end{equation}
As $e \neq 0$, we see that $\|e\|^{2}_{E} >0$. Since $\theta \in [0, 2\pi)$ is arbitrary, it must be 
\begin{equation}\label{0.51}
e^{*}\left(\frac{T(1 \otimes e)(x)-T(1 \otimes e)(y)}{d(x,y)}\right)=0
\end{equation}
for any $e^{*} \in \es$. This implies 
\[
\frac{T(1 \otimes e)(x)-T(1 \otimes e)(y)}{d(x,y)}=0
\]
for any $(x,y) \in \wx$. Thus we deduce 
$L(T(1\otimes e))=0$.
Therefore, there exists $e_0 \in E$ such that $T(1\otimes e)=1 \otimes e_0$.
\end{proof}

Applying Lemma \ref{tako} 
we 
define a map $\phi: E \to E$ by 
\begin{equation}\label{takotako}
T(1 \otimes e)=1 \otimes \phi(e) 
\end{equation}
 for each $e \in E$. 
By (\ref{0.5}) and (\ref{0.51}) we have $\mathbb{R} \ni [T(1 \otimes e), 1\otimes e]_{\Psi_{\theta}L}=[T(1 \otimes e)(x'),e]_{E}$. This implies that 
$[\phi(e),e]_{E} \in \mathbb{R}$ for any $e \in E$. Since $T$ is a bounded linear operator, we get $\phi$ is a hermitian operator on $E$.

We give a complete description of hermitian operators on $\lxe$ with $\|\cdot\|_{L}$, where $E$ is any complex Banach space (without assuming that $E$ is of a finite dimension). 
\begin{theorem}\label{hermitian}
Let $X$ be a compact metric space and $E$ a complex Banach space. Then $T: (\lxe,\|\cdot\|_{L}), \to (\lxe,\|\cdot\|_{L})$ is a hermitian operator
 if and only if there exists a hermitian operator $\phi: E \to E$ such that
\begin{equation}\label{he1}
TF(x)=\phi(F(x)),  \quad F \in \lxe, \quad  x \in X.
\end{equation}
\end{theorem}

\begin{proof}[Proof of Theorem \ref{hermitian}]
Suppose that $T$ is of the form described as (\ref{he1}) in the statement of Theorem \ref{hermitian}. 
To prove that $T$ is a hermitian operator, we apply the fact that $T$ is a hermitian if and only if $e^{itT}$ is a surjective isometry for every $t \in \mathbb{R}$, see \cite[Theorem 5.2.6]{FJB}. 
Let $t \in \mathbb{R}$. By the definition of $T$, we have
\[
e^{itT}F(x)=e^{it\phi}(F(x))
\]
for any $F \in \lxe$ and $x \in X$. Since $\phi$ is a hermitian on $E$, $e^{it\phi}$ is a surjective isometry.
This implies that 
$\|e^{itT}F\|_{\infty}=\|F\|_{\infty}$
and 
$L(e^{itT}F)=L(F).$
Thus we deduce
$\|e^{itT}F\|_{L}=\|F\|_{L}$
for any $F \in \lxe$. Since $e^{itT}$ is a surjective isometry for every $t \in \mathbb{R}$,  we conclude $T$ is a hermitian operator.


We prove the converse. 
Suppose that $T:\lxe \to \lxe$ is a hermitian operator. Let $\phi$ be the operator defined by \eqref{takotako}.  A similar argument with above yields
an operator from $\lxe$ into itself  given by $F \mapsto \phi\circ F$ is a hermitian operator. Hence we define a hermitian operator $\Tb: \lxe \to \lxe$ by
\[
(\Tb F)(x)=(TF)(x)-\phi(F(x))
\]
for all $F \in \lxe$ and $x \in X$. We shall prove that $\Tb=0$ on $\lxe$ in two steps. \\
\underline{{\bf{Step1.}} For any $f \in \Lip(X)$ and $e \in E$ we have $\Tb(f \otimes e)=0$.} \\
Note that the same idea with \cite[Theorem 2.2]{hermitian1} is valid even if we replace a finite dimensional Banach space $E$ with a Banach space $E$. 

By \cite[p. 10]{FJB} there is a semi-inner product $[\cdot,\cdot]_{E}$ on $E$ compatible with the norm such that  $[e_1,\lambda e_2]_{E}=\bar\lambda[e_1,e_2]_{E}$ for any $e_i \in E$ and $\lambda \in \mathbb{C}$.
Let $e \in \mathbb{S}(E)$. We define a map $S_{e}: \lx \to \lx$ by
\[
S_{e}(f)(x)=[\Tb (f \otimes e)(x),e]_{E}, \quad f \in \lx,\,\, x \in X.
\]
By simple calculations we have $S_e$ is a bounded linear operator with $\|S_e\| \le \|\Tb\|$. 
Moreover we shall prove $S_e$ is a hermitian operator. Let $t \in \mathbb{R}$. By the definition of $S_{e}$ we get $(I+itS_{e})(1)(x)=1$
for any $x \in X$. This implies that 
\begin{equation}\label{0.6}
1 \le \|I+itS_{e}\|.
\end{equation}
On the other hand, let $f \in \lx$. We obtain for any $x, y \in X$,
\begin{equation*}
|(I+itS_{e})(f)(x)|\le \|(I+it\Tb)(f \otimes e)\|_{\infty}
\end{equation*}
and
\begin{equation*}
|(I+itS_{e})(f)(x)-(I+itS_{e})(f)(y)| \le L((I+it\Tb)(f \otimes e))d(x,y).
\end{equation*}
Therefore, we get
\begin{equation*}
\begin{split}
\|(I+itS_{e})(f)\|_{L}
&\le  \|(I+it\Tb)(f \otimes e)\|_{\infty}+L((I+it\Tb)(f \otimes e))\\
&\le \|I+it\Tb\|\|f\otimes e\|_{L}=\|I+it\Tb\|\|f\|_{L}
\end{split}
\end{equation*}
for any $f \in \lx$. We conclude that
\begin{equation}\label{0.7}
\|I+itS_{e}\| \le \|I+it\Tb\|.
\end{equation}
Since  $\Tb$ is a hermitian operator on $\lxe$, we have 
$\|I+it\Tb\|=1+o(t)$
by \cite[Theorem 5.2.6]{FJB}. By (\ref{0.6}) and (\ref{0.7}), we see that
\[
1 \le \|I+itS_{e}\| \le \|I+it\Tb\|=1+o(t).
\]
This implies that $S_{e}: \lx \to \lx$ is a hermitian operator. By \cite[Theorem 3.1.]{BJJMs} we have $S_e$ is a real multiple of the identity. Since $S_{e}(1)(x)=[\Tb(1\otimes e)(x), e]=0$, we deduce
$S_{e}(f)(x)=0f(x)=0$  for any $f \in \lx$ and $x \in X$.
This implies that
$[\Tb(f \otimes e)(x), e]_{E}=0$
for all $f \in \lx$ and $x \in X$. As $e \in \mathbb{S}(E)$ is arbitrary, we obtain
\begin{equation}\label{0.8}
[\Tb(f \otimes e)(x), e]_{E}=0, \quad  e \in E, \quad f \in \lx, \quad x \in X.
\end{equation}
 Let $f \in \lx$ and $x \in X$. Then we define a map $\Sf: E \to E$ by
$\Sf(e)=\Tb(f \otimes e)(x)$
for any $e \in E$. Since $\Tb$ is a bounded linear operator,  $\Sf$ is also a bounded linear operator with $\|\Sf\| \le \|\Tb\|\|f\|_{L}$. 
By (\ref{0.8}) we have
$[\Sf(e),e]_{E}=[\Tb(f \otimes e)(x), e]_{E}=0$
for all $e \in E$. Applying \cite[Theorem 5]{Lu}, we have
$\Tb(f \otimes e)(x)=\Sf(e)=0$ for any $e \in E$.
As $f \in \lx$ and $x \in X$ are arbitrary, we conclude Step 1.\\
\underline{{\bf{Step 2.}} For any $F \in \lxe$, we have $\Tb(F)=0$.}\\
If $F \in \Lip(X) \otimes E$, Step1 yields that  $\Tb(F)=0$ by the linearity of $\Tb$. Thus it suffices to show $\Tb(F)=0$ holds for any $F \in \lxe \setminus \Lip(X) \otimes E$. 
Let $F \in \lxe \setminus \Lip(X) \otimes E$ with $F(x_0)=0$.  For any $e \in \mathbb{S}(E)$, put 
\[
G_e=(\|F\|_{\infty}-|F|)\otimes e +F,
\]
where $|F|(x):=\|F(x)\|$ and $|F| \in \Lip(X)$. 
Then we have 
\[
G_{e}(x_0)=\|F\|_{\infty}e
\]
 and 
\begin{multline*}
\|G_{e}(x)\|=\|(\|F\|_{\infty}-\|F(x)\|)\otimes e +F(x)\| \\ \le \|F\|_{\infty}-\|F(x)\|+\|F(x)\|  = \|F\|_{\infty}
\end{multline*}
for any $x \in X$. Thus we obtain $\|G_e(x_0)\|=\|F\|_{\infty}=\|G_e\|_{\infty}$. As $\beta(\wx \times \es) $ is compact, there are $\xi \in \beta(\wx \times \es) $  and $\alpha \in \mathbb{C}$ with $|\alpha|=1$ such that $\alpha\widetilde{G_e}(\xi)=L(G_e)$. This implies that $(x_0, \xi, \alpha e) \in P_{G_e}$. We choose a choice function $\Psi_{e}:\lxe \to X \times \beta(\wx \times \es) \times \mathbb{B}(E)$ such that $\Psi_e(G_e)=(x_0, \xi, \alpha e)$ and define a semi-inner product $[\cdot,\cdot]_{\Psi_e L}$ in the manner as in (\ref{choice}). Since $T_0: \Lip(X,E) \to \Lip(X,E)$ is a hermitian operator, we get 
\begin{multline*}
\mathbb{R} \ni [T_0(G_e),G_e]_{\Psi_e L}= [T_0(F),G_e]_{\Psi_e L}\\
= [T_0(F)(x_0)+\alpha \widetilde{T_0(F)}(\xi)e,\|F\|_{\infty}e+L(G_e)e]_{E}\\=(e^{*}(T_0(F)(x_0))+\alpha \widetilde{T_0(F)}(\xi))\|G_{e}\|_{L},
\end{multline*}
where $e^{*} \in \es$ with $e^{*}(e)=1$ for any $e \in \mathbb{S}(E)$. We have
\begin{equation}\label{one}
e^{*}(T_0(F)(x_0))+ \alpha \widetilde{T_0(F)}(\xi) \in \mathbb{R}.
\end{equation}
On the other hand, there exists $y_0 \in X$ such that $\|F(y_0)\|=\|F\|_{\infty}\neq 0$ and there is $f_{y_0} \in \mathbb{S}(E)$ such that $F(y_0)=\|F\|_{\infty}f_{y_0}$. 
We get
\[
G_e(y_0)=F(y_0)=\|F\|_{\infty}f_{y_0}.
\]
This implies that $\|G_e(y_0)\|=\|F(y_0)\|=\|F\|_{\infty}=\|G_{e}\|_{\infty}$. We have 
\begin{multline*}
\|\Gamma G_{e}(y_0,\xi,\alpha f_{y_0})\|_{E}=\|G_e(y_0)+\alpha \widetilde{G_e}(\xi)f_{y_0}\|_{E}\\
=\|\|F\|_{\infty}f_{y_0}+L(G_e)f_{y_0}\|_{E}=\|F\|_{\infty}+L(G_e)=\|G_e\|_{L}.
\end{multline*}
Thus we get $(y_0, \xi, \alpha f_{y_0}) \in P_{G_{e}}$. In the same manner,  there is a choice function $\Psi_{f_{y_0}}:\lxe \to X \times \beta(\wx \times \es) \times \mathbb{B}(E)$ such that $\Psi_{f_{y_0}}(G_e)=(y_0, \xi, \alpha f_{y_0}) $ and we can define a semi-inner product $[\cdot,\cdot]_{\Psi_{f_{y_0}}L}$ on $\lxe$.  It follows that 
\begin{multline*}
\mathbb{R} \ni [T_0(G_e),G_e]_{\Psi_{f_{y_0}}L}= [T_0(F),G_e]_{\Psi_{f_{y_0}}L}\\
= [T_0(F)(y_0)+\alpha \widetilde{T_0(F)}(\xi)f_{y_0},\|G_{e}\|_{L}f_{y_0}]_{E}\\=({f_{y_0}}^{*}(T_0(F)(y_0))+\alpha \widetilde{T_0(F)}(\xi))\|G_{e}\|_{L},
\end{multline*}
where ${f_{y_0}}^{*}  \in \es$ with ${f_{y_0}}^{*} (f_{y_0})=1$. We obtain
\begin{equation}\label{two}
{f_{y_0}}^{*}(T_0(F)(y_0))+\alpha \widetilde{T_0(F)}(\xi) \in \mathbb{R}.
\end{equation}
By (\ref{one}) and (\ref{two}) we get $e^{*}(T_0(F)(x_0))-{f_{y_0}}^{*}(T_0(F)(y_0)) \in \mathbb{R}$. Since $e \in \mathbb{S}(E)$ is arbitrary, it follows that $T_0F(x_0)=0$. Let $F \in \lxe \setminus \Lip(X) \otimes E$ and $x \in X$. We define $F_x=F-1 \otimes F(x)$. Since $F_x(x)=0$, we get 
\[
0=T_0(F_x)(x)=T_0F(x)-T_0(1 \otimes F(x))(x)=T_0(F)(x).
\]
Thus we have  $\Tb(F)=0$ for any $F \in \lxe$ and conclude Step 2.

Therefore we obtain $TF(x)=\phi(F(x))$ for any $F \in \lxe$. This completes the proof.
\end{proof}

\section{An extension of  isometries on $\Lip(X) \otimes E$}\label{33333}
We define the notation of T-sets which is introduced by Myers in \cite{My}.  
\begin{definition}
Let $(A,\|\cdot\|_{A})$ be a semi-normed space. For a subset $\mathbb{U}$ of $A$, we call $\mathbb{U}$ a T-set of $A$ with respect to $\|\cdot\|_{A}$ if $\mathbb{U}$ satisfies the property that for any finite collection $a_1, \cdots, a_n \in \mathbb{U}$, $\|\Sigma^{n}_{i=1}a_i\|_{A}=\Sigma^{n}_{i=1}\|a_i\|_{A}$ and such that $\mathbb{U}$ is a maximal with respect to the property. If no confusion is possible, we will refer to T-set of $A$  with respect to $\|\cdot\|_{A}$ as T-set of $A$.
\end{definition}

\begin{lemma}\label{mreal}
Let $(A,\|\cdot\|_{A})$ be a Banach space and $\mathbb{U}$ a T-set of $A$ with respect to $\|\cdot\|_{A}$. If $a \in \mathbb{U}$ then $\lambda a \in \mathbb{U}$ for any $\lambda \ge 0$.
\end{lemma}
\begin{proof}
We conclude this Lemma by the Hahn$-$Banach theorem immediately.
\end{proof}

\begin{lemma}\label{isop}
Let $N_i$ be normed spaces for $i=1,2$. Suppose that $U:N_1 \to N_2$ is a surjective isometry with $U(0)=0$. Then $U$ maps T-set of $N_1$ to T-set of  $N_2$.
\end{lemma}
\begin{proof}
It follows from the Mazur$-$Ulam theorem that every surjective isometry $U$ between two normed spaces with $U(0)=0$ is a real linear isometry. By the maximality of T-set and surjectivity of $U$, we conclude that $U$ preserves a T-set.
\end{proof}

Let $(E, \|\cdot\|_{E})$ be a Banach space. Let $x \in X$, $\mathbb{U}$ be  a T-set of $E$ with respect to $\|\cdot\|_{E}$ and $\mathbb{T}$ be  a T-set of $\Lip(X,E)$ with respect to $L(\cdot)$. We write 
\[
S(x,\mathbb{U},\mathbb{T})=\{F \in \Lip(X,E)|F(x)\in \mathbb{U}, \|F(x)\|_{E}=\|F\|_{\infty},F \in \mathbb{T}\}.
\]

\begin{lemma}\label{3.4m}
 Let $x \in X$, $\mathbb{U}$ be a T-set of $E$ and $\mathbb{T}$ be a T-set of $\Lip(X,E)$ with respect to $L(\cdot)$.  Then for any finite collection $F_1,\cdots,F_n \in S(x,\mathbb{U},\mathbb{T})$, we have $\|\Sig F_i\|_{L}=\Sig\|F_i\|_{L}$.
\end{lemma}

\begin{proof}
For any $F_1,\cdots,F_n \in S(x,\mathbb{U},\mathbb{T})$, we have $F_i(x) \in \mathbb{U}$ and $\|F_{i}(x)\|_{E}=\|F_i\|_{\infty}$ for any $i=1,\cdots,n$. We get 
\[
\|\Sig F_i\|_{\infty} \le \Sig \|F_i\|_{\infty} =\Sig \|F_i(x)\|=\|\Sig F_i(x)\| \le \|\Sig F_i\|_{\infty}.
\]
This implies that 
$\|\Sig F_i\|_{\infty}=\Sig\|F_i\|_{\infty}$.
Since $F_i \in \mathbb{T}$ for any $i=1,\cdots,n$, we also get
$L(\Sig F_i)=\Sig L(F_i)$.
This implies that  
$\|\Sig F_i\|_{L}=\Sig\|F_i\|_{L}$.
\end{proof}

\begin{prop}\label{3.1}
Let $\mathcal{S}$ be a T-set of $\lxe$ with respect to $\|\cdot\|_{L}$. Then there is $x \in X$ and there are $\mathbb{U}$ and  $\mathbb{T}$, where $\mathbb{U}$ is a T-set of $E$ and $\mathbb{T}$ is a T-set of $\lxe$ with respect to $L(\cdot)$, such that $\mathcal{S}=S(x,\mathbb{U},\mathbb{T})$.
\end{prop}
\begin{proof}
For any $F \in \mathcal{S}$ we write $P(F):=\{x \in X| \|F(x)\|_{E}=\|F\|_{\infty}\}$. We shall show that $\bigcap_{F \in \mathcal{S}}P(F)\neq \emptyset$. For any finite collection $F_1,\cdots,F_n \in \mathcal{S}$, since $\|\Sig F_i\|_{L}=\Sig\|F_i\|_{L}$ we have $\|\Sig F_i\|_{\infty}=\Sig\|F_i\|_{\infty}$. Since $\Sig F_i \in \lxe$, there is $x \in X$ such that  $\|(\Sig F_i)(x)\|=\|\Sig F_i\|_{\infty}$. Thus we get 
\[
\Sig \|F_i\|_{\infty}=\|\Sig F_i\|_{\infty}=\|(\Sig F_i)(x)\| \le \Sig\|F_i(x)\|.
\]
This implies that $\|F_i\|_{\infty}=\|F_i(x)\|$ for any $i=1,\cdots,n$ and $x \in \bigcap^{n}_{i=1}P(F_i)$. Since $X$ is compact and $P(F)$ is a closed set for each $F \in \mathcal{S}$, we have $\bigcap_{F \in \mathcal{S}}P(F) \neq \emptyset$ by the finite intersection property.

Let $x \in \bigcap_{F \in \mathcal{S}}P(F) $. We consider the set $R_{x}(\mathcal{S}):=\{F(x) \in E| F \in \mathcal{S}\}$. Choose any finite collection $F_1(x),\cdots,F_n(x) \in R_{x}(\mathcal{S})$. Since $\Sig F_i \in \mathcal{S}$, we have $x \in P(\Sig F_i )$. This implies
\[
\Sig \|F_i\|_{\infty}=\|\Sig F_i\|_{\infty}=\|\Sig F_i(x)\| \le \Sig\|F_i(x)\|=\Sig \|F_i\|_{\infty}.
\]
Thus we have $\|\Sig F_i(x)\| = \Sig\|F_i(x)\|$, which means that there is a T-set $\mathbb{U}$ of $E$ such that $R_{x}(\mathcal{S}) \subset \mathbb{U}$. Therefore for any $F \in \mathcal{S}$ we have $F(x) \in \mathbb{U}$ and  $\|F(x)\|=\|F\|_{\infty}$. 

Since  $L(\Sig F_i)=\Sig L(F_i)$ for any finite collection $F_1,\cdots,F_n \in \mathcal{S}$, there exists a T-set $\mathbb{T}$ of $\lxe$ with respect to $L(\cdot)$ such that $\mathcal{S} \subset \mathbb{T}$. This implies that $\mathcal{S} \subset S(x,\mathbb{U},\mathbb{T})$. By Lemma \ref{3.4m} and maximality of $\mathcal{S}$, we conclude that $\mathcal{S}=S(x,\mathbb{U},\mathbb{T})$.
\end{proof}

\begin{prop}\label{3.6}
Let $X_i$ be a compact metric space and $E_i$ be a Banach space for $i=1,2$. Let $U:\Lip(X_1,E_1) \to \Lip(X_2,E_2)$ be a surjective complex linear isometry. Suppose that there is a surjective complex linear isometry $\psi:E_1\to E_2$ and there is a surjective isometry $\varphi:X_2 \to X_1$ such that 
$U(f \otimes e)(y)=\psi(f(\varphi(y))e)$
for any $f \in \Lip(X_1)$ and $e \in E_1$. Then 
\[
UF(y)=\psi(F(\varphi(y)))
\]
 for any $F \in \Lip(X_1,E_1)$.
\end{prop}

In the rest of this section, we assume that a surjective complex linear isometry $U:\Lip(X_1,E_1) \to \Lip(X_2,E_2)$ satisfies the assumption of Proposition \ref{3.6}. To prove Proposition \ref{3.6} we first show the following lemma.
\begin{lemma}\label{zerotozero}
Let $x_0 \in X_1$ and $F \in \lxo$ with $\|F\|_{\infty}=1$ and $F(x_0)=0$. Then $UF(y_0)=0$, where $y_0=\varphi^{-1}(x_0)$.
\end{lemma}

\begin{proof}
Suppose that $UF(y_0)\neq 0$. Put $a=\frac{UF(y_0)}{\|UF(y_0)\|}$. 
The map from $\mathbb{S}(E_{2})$ to $\mathbb{R}$ defined by $e \mapsto \|UF(y_0)+(\|UF\|_{\infty}+1)e\|$ is continuous.
 Since $\|UF(y_0)\| \neq 0$, we have 
\begin{multline*}
\|UF(y_0)+(\|UF\|_{\infty}+1)a\|=\|\frac{UF(y_0)}{\|UF(y_0)\|}(\|UF(y_0))\|+\|UF\|_{\infty}+1)\|\\= \|UF(y_0))\|+\|UF\|_{\infty}+1 >\|UF\|_{\infty}+1.
\end{multline*}
 There exists $\delta >0$ such that if $e \in \mathbb{S}(E_2)$ with $\|a-e\|<\delta$ then  $\|UF(y_0)+(\|UF\|_{\infty}+1)e\|>\|UF\|_{\infty}+1$. 
We choose $\theta \in (0,2\pi)$ such that $|e^{i\theta}-1|<\delta$. We write $\et:=\psi^{-1}(e^{i\theta}a)$. This implies that 
\begin{equation}\label{kaiten}
\|UF(y_0)+(\|UF\|_{\infty}+1)\psi(\et)\|>\|UF\|_{\infty}+1.
\end{equation}
For any $n \in \mathbb{N}$, we define $g_n \in \Lip(X_1)$ by 
\[
g_n(x)=(\|UF\|_{\infty}+1)\max\{1-nL(F)d(x,x_0),0\}, \quad x \in X_1.
\]
There is $\mathcal{S}_n$ which is a T-set  of $\lxo$ with respect to $\|\cdot\|_{L}$ such that $F+g_n\otimes \et \in \mathcal{S}_n$. Since we have 
\[
(F+g_n\otimes \et)(x_0)=(\|UF\|_{\infty}+1)\et.
\]
When $x \neq x_0$ and  $1-nL(F)d(x,x_0)\ge 0$, we have 
\begin{multline*}
\|(F+g_n\otimes \et)(x)\|=\|F(x)-F(x_0)+(g_n\otimes \et)(x)\|\\ \le L(F)d(x,x_0)+(\|UF\|_{\infty}+1)(1-nL(F)d(x,x_0))\\=(1-n(\|UF\|_{\infty}+1)) L(F)d(x,x_0)+\|UF\|_{\infty}+1< \|UF\|_{\infty}+1.
\end{multline*}
When $1-nL(F)d(x,x_0)\le 0$, we have 
\[
\|(F+g_n\otimes \et)(x)\|=\|F(x)\|\le 1< \|UF\|_{\infty}+1.
\]
Thus we obtain 
 $P(F+g_n\otimes \et )=\{x_0\}$. 
By Proposition \ref{3.1} there are T-set $\mathbb{U}_n \subset E_1$ and T-set $\mathbb{T}_n \subset \lxo$ such that $ F+g_n\otimes \et \in \mathcal{S}_n=S(x_0,\mathbb{U}_n,\mathbb{T}_n)$. In particular we have
\[
(\|UF\|_{\infty}+1)\et=(F+g_n\otimes \et)(x_0)  \in \mathbb{U}_n.
\]
By Lemma \ref{mreal}, $\et \in \mathbb{U}_n$. 
Since $U$ is a surjective isometry, Lemma \ref{isop} shows that there are $y_n \in X_2$, T-set $\mathbb{V}_n \subset E_2$ and T-set $\mathbb{T'}_n \subset \lxt$ with respect to $L(\cdot)$ such that $U(S(x_0,\mathbb{U}_n,\mathbb{T}_n))=S(y_n,\mathbb{V}_n,\mathbb{T'}_n)$. Since $\et \in \mathbb{U}_n$, we have $1 \otimes \et \in S(x_0,\mathbb{U}_n,\mathbb{T}_n) $. By the assumption, 
 we have $U(1 \otimes \et)=1 \otimes \psi(\et) \in S(y_n,\mathbb{V}_n,\mathbb{T'}_n)$. It implies that $\psi(\et) \in \mathbb{V}_n$ for any $n \in \mathbb{N}$.
For any $y \in X_2$
\begin{multline*}
U(F+g_n \otimes \et)(y)=UF(y)+\psi(g_n(\varphi(y))\et)\\=UF(y)+(\|UF\|_{\infty}+1)\max\{1-nL(F)d(\varphi(y),x_0),0\}\psi(\et).
\end{multline*}
We shall show that the sequence $\{y_n\}$ converges $y_0$ as $n \to \infty$. Suppose that there exists $n \in \mathbb{N}$ such that $1-nL(F)d(\varphi(y_n),x_0) < 0$. Then we have 
\[
\|UF\|_{\infty} \ge \|UF(y_n)\|=\|(UF+U(g_n\otimes \et))(y_n)\|=\|UF+U(g_n\otimes \et)\|_{\infty}
\]
and 
\begin{multline*}
\|UF+U(g_n\otimes \et)\|_{\infty} \ge \|UF(y_0)+(\|UF\|_{\infty}+1)\psi(\et)\|>\|UF\|_{\infty}+1.
\end{multline*}
by (\ref{kaiten}). This is a contradiction. Thus for every $n \in \mathbb{N}$ we have 
\[
1-nL(F)d(\varphi(y_n),x_0) \ge 0.
\]
Thus we get $\frac{1}{nL(F)}>d(\varphi(y_n),x_0)=d(\varphi(y_n),\varphi(y_0))=d(y_n,y_0)$. This implies that $y_n \to y_0$ as $n \to \infty$. Since $UF \in \lxt$ we get $UF(y_n) \to UF(y_0)$. 
 
Because we obtain $0\le 1-nL(F)d(\varphi(y_n),x_0) \le 1$, the sequence $\{ 1-nL(F)d(\varphi(y_n),x_0)\}$ has a convergent subsequence. Without loss of generality we can assume that the sequence converges to $\beta \in [0,1]$ as $n \to \infty$. We write 
\[
c_n:=U(F+g_n\otimes \et)(y_n)=UF(y_n)+(\|UF\|_{\infty}+1)(1-nL(F)d(\varphi(y_n),x_0))\psi(\et)
\]
 and 
\begin{equation}\label{taukabuseru}
c_0:=UF(y_0)+(\|UF\|_{\infty}+1)\beta \psi(\et).
\end{equation}
We obtain that 
\begin{equation}\label{NS}
\|c_n-c_0\|_{E_2} \to 0  \ \quad \text{if} \quad n \to \infty.
\end{equation} 
As $U(F+g_n\otimes \et) \in S(y_n,\mathbb{V}_n,\mathbb{T'}_n)$, we get $c_n \in \mathbb{V}_n$.
Since $\psi(\et) \in \mathbb{V}_n$, we have $\|c_n+\psi(\et)\|=\|c_n\|+\|\psi(\et)\|$. By (\ref{NS}) we get $\|c_0+\psi(\et)\|=\|c_0\|+\|\psi(\et)\|$. As $\psi(\et)=e^{i\theta}a$ we obtain
\[
\|e^{-i\theta}c_0+a\|=\|e^{-i\theta}c_0\|+\|a\|.
\]
Thus there is $\tau \in E_2^{*}$ such that $\|\tau\|=1$, $\tau(e^{-i\theta}c_0)=\|c_0\|$ and $\tau(a)=\|a\|=1$. By (\ref{taukabuseru}) and $a=\frac{UF(y_0)}{\|UF(y_0)\|}$, we have 
\begin{multline*}
e^{i\theta}\|c_0\|=\tau(c_0)=\tau(UF(y_0))+\tau((\|UF\|_{\infty}+1)\beta e^{i\theta}a)\\
=\|UF(y_0)\|+e^{i\theta}(\|UF\|_{\infty}+1)\beta. 
\end{multline*}
We obtain that $\|UF(y_0)\|=e^{i\theta}(\|c_0\|-(\|UF\|_{\infty}+1)\beta)$. As $\theta \in (0,2\pi)$ and $\|c_0\|-(\|UF\|_{\infty}+1)\beta \in \mathbb{R}$ , we conclude $UF(y_0)=0$.
\end{proof}

\begin{proof}[The proof of Proposition \ref{3.6}]
By the assumption, it suffices to show that $UF(y)=\psi(F(\varphi(y)))$ holds for any $F \in \lxo$ which $F$ is not a constant map. For any $x \in X_1$, we define $G:=F-1 \otimes F(x)$. Then, we have $G(x)=0$.  As $G \neq 0$, without loss of generality, we assume that $\|G\|_{\infty}=1$. By Lemma \ref{zerotozero}, we obtain $(UG)(\varphi^{-1}(x))=0$. This implies that $(UF)(\varphi^{-1}(x))=U(1 \otimes F(x))=\psi(1(x)F(x))=\psi(F(x))$. 
\end{proof}


\begin{section}{Proof of Theorem \ref{isometrydd}}\label{44444}
Let $\ma_i$ be unital $C^{*}$-algebras for $i=1,2$. 
In this section we consider unital surjective complex isometries with respect to the norm $\|\cdot\|_L$ from $\lxom$ onto $\lxtm$.
Although we  apply the similar arguments as \cite{hermitian1}, we show a proof without omitting it because this is a generalization for \cite[Theorem 3.3]{hermitian1}. 
We say that a bounded operator $D$ on a unital $C^*$-algebra $\mathcal A$ is a $*$-derivation if 
\begin{equation}\label{star}
\begin{split}
&D(ab)=D(a)b+aD(b),
\\
& D(a^*)=D(a)^*\\
\end{split}
\end{equation}
for every pair $a, b\in \mathcal{A}$. 
By the definition, it is easy to see that $D(1)=1$ for any  $*$-derivation on $\mathcal A$. 
For each $a \in \mathcal{A}$, a left multiplication operator $M_a:\mathcal{A}\to \mathcal{A}$ is defined by $M_ab=ab$ for every $b\in \mathcal{A}$. We denote the set of all hermitian elements of $\ma$ by $\hm$.

The following is the characterization of hermitian operators on a unital $C^*$-algebra.
\begin{theorem}[Sinclair \cite{Si}]\label{Si}
Let $\mathcal{A}$ be a unital $C^*$-algebra. A bounded operator $T$ on $\mathcal{A}$ is a hermitian operator  if and only if there exist  $h\in \hm$ and a $*$-derivation $D$ on $\mathcal{A}$ such that 
$T=M_h+ iD$.
\end{theorem}

We introduce the notation to characterize hermitian operators on $\lxm$.
\begin{definition}\label{hel}
For any $h \in \hm$, we define a multiplication operator $M_{1\otimes h}: \lxm \to \lxm$ by
\[
M_{1\otimes h}(F)=(1\otimes h)F,  \quad F \in \lxm.
\]
For any $*$-derivation $D: \ma \to \ma$, we define a map $\widehat{D}: \lxm \to \lxm$ by
\[
\widehat{D}(F)(x)=D(F(x)), \quad F \in \lxm ,  \quad x \in X.
\] 
\end{definition}
Combining Theorem \ref{Si} and Theorem \ref{hermitian} we obtain the following.
\begin{prop}\label{0}
Suppose that $T: \lxm \to \lxm$ is a map. Then $T$ is a hermitian operator if and only if there exists $h \in \hm$ and a $*$-derivation $D$ on $\mathcal A$ such that
\begin{equation}\label{hel1}
T=M_{1\otimes h}+i\widehat{D}.
\end{equation}
\end{prop}



The following Proposition  is a well-known fact. 
\begin{prop}\label{iti}
Let $\mathcal{B}_j$ be Banach algebras for $j=1,2$.  Suppose that $U$ is a surjective complex linear isometry from $\mathcal{B}_1$ onto $\mathcal{B}_2$ and $T$ is a hermitian operator on $\mathcal{B}_1$. Then the map $UTU^{-1}$ is a hermitian operator on $\mathcal{B}_2$.
\end{prop}

In the rest of this section we consider  a surjective complex linear isometry $U:(\lxom, \|\cdot\|_{L}) \to (\lxtm, \|\cdot\|_{L})$ with $U(1)=1$.

\begin{lemma}\label{1}
For any $h \in \hmo$, there exists $h' \in \hmt$ such that
\[
U(1\otimes h)=1\otimes h'.
\]
\end{lemma}

\begin{proof}
Let $h \in \hmo$.  By Proposition \ref{0} and Proposition \ref{iti}, $UM_{1 \otimes h}U^{-1}$ is a hermitian operator on $\lxtm$.  Thus there exists $h' \in \hmt$ and a $*$-derivation $D$ on $\mathcal{A}_2$ such that
$UM_{1 \otimes h}U^{-1}=M_{1 \otimes h'}+i\widehat{D}$. For any $y \in X_2$, we have 
\[
(UM_{1 \otimes h}U^{-1})(1)(y)=(UM_{1 \otimes h}1)(y)=U(1 \otimes h)(y)
\]
and
\[
h'(1(y))+iD(1(y))=h'+i0=h'.
\]
This implies that 
 $U(1\otimes h)=1\otimes h'$. 
\end{proof}
By Lemma \ref{1}, we define a map $\psi_0: \hmo \to \hmt$ by
\[
U(1 \otimes h)=1 \otimes \psi_0(h).
\]

\begin{lemma}\label{2}
The map $\psi_0$ is a real linear isometry from $\hmo$ onto $\hmt$ 
such that $\psi_0(1)=1$.
\end{lemma}
\begin{proof}
For any $h_2 \in \hmt$, we have that $U^{-1}M_{1 \otimes h_2}U$ is a hermitian operator on $\lxom$. By Proposition \ref{0}  there are $h_1 \in \hmo$ and a $*$-derivation $D_1$ on $\ma_1$ such that
\[
U^{-1}M_{1 \otimes h_2}U=M_{1 \otimes h_1}+i\widehat{D_1}.
\]
Since  we have $M_{1\otimes h_1}=U^{-1}M_{1 \otimes h_2}U-i\widehat{D_1}$, we get 
\begin{multline*}
UM_{1\otimes h_1}U^{-1}(1)=
U(U^{-1}M_{1 \otimes h_2}U-i\widehat{D_1})U^{-1}(1)\\
=M_{1 \otimes h_2}(1)-U(i\widehat{D_1}(1))
=1 \otimes h_2-iU(0)
=1 \otimes h_2.
\end{multline*}
We obtain $U(1 \otimes h_1)=1 \otimes h_2$ and $\psi_0(h_1)=h_2$. It follows that  $\psi_0$ is surjective. For any $h \in \hmo$, we get 
$\|\psi_0(h)\|=\|1 \otimes \psi_0(h)\|_{L}=\|U(1 \otimes h)\|_{L}
=\|1 \otimes h \|_{L}=\|h\|$.
Thus, we have $\psi_0$ is an isometry. Since $U$ is a linear map, it is easy to see that $\psi_0$ is real linear.
Moreover $U(1)=1$, we get  $\psi_0(1)=1$.
\end{proof}

For any $a \in \mathcal{A}_1$, there are $h_1, h_2 \in \hmo$ such that $a=h_1+i h_2$. 
Thus, we define a map $\psi: \mathcal{A}_1 \to \mathcal{A}_2$ by
\[
\psi(a)=\psi(h_1+ih_2):=\psi_0(h_1)+i\psi_0(h_2).
\]
By a simple calculation, we have
\begin{equation}\label{psi}
U(1 \otimes a)=1 \otimes \psi(a)
\end{equation}
for any $a \in \mathcal{A}_1$.
\begin{lemma}\label{3}
The map $\psi$ is a surjective complex linear isometry from $\mathcal{A}_1$ onto $\mathcal{A}_2$
such that $\psi(1)=1$.
\end{lemma}

\begin{proof}
By (\ref{psi}), we have $\psi$ is a complex linear isometry with $\psi(1)=1$. Therefore it suffices to show  $\psi$ is  surjective.  For any $a \in \mathcal{A}_2$, there exists $h_1$, $h_2 \in \hmt$ such that $a=h_1+ih_2$. Since Lemma \ref{2} shows that $\psi_0: \hmo \to \hmt$ is  surjective, there are $h^{'}_1, h^{'}_2 \in \hmo$ such that $\psi_0(h^{'}_1)=h_1$ and $\psi_0(h^{'}_2)=h_2$. Then we get $a'=h^{'}_1+ih^{'}_2 \in \mathcal{A}_1$. This implies that 
\[
\psi(a')=\psi_0(h^{'}_1)+i\psi_0(h^{'}_2)=h_1+ih_2=a.
\]
This completes the proof.
\end{proof}

\begin{lemma}\label{4}
Suppose that $\ma_i$ is a unital factor $C^{*}$-algebra for $i=1,2$. Then 
there exists a surjective isometry $\varphi: X_2 \to X_1$ such that
\[
U(f \otimes 1)(y)=f(\varphi(y))\otimes 1
\]
for all $f \in \Lip(X_1)$ and $y \in X_2$.
\end{lemma}

\begin{proof}
For any $b \in \mathcal{A}_2$ with $b^{*}=-b$, we define a $*$-derivation $D$ on $\mathcal{A}_2$ by
\[
D(a)=ba-ab, \quad a \in \ma_1.
\]
Note that Proposition \ref{0} shows that the map $i\widehat{D}: \lxtm \to \lxtm$ defined by
\[
(i\widehat{D})(F)(y)=iD(F(y)) \quad F \in \lxtm, \quad y \in X_2,
\]
is a hermitian operator on $\lxtm$. Since the map $U$ is an isometry, $U^{-1}i\widehat{D}U$ is a hermitian operator on $\lxom$. By Proposition \ref{0} there exists  $h \in \hmo$ and $*$-derivation $D'$ on $\ma_1$ such that 
\[
U^{-1}i\widehat{D}U=M_{1 \otimes h}+i\widehat{D'}.
\]
As $U(1)=1$, we get
\[
(U^{-1}i\widehat{D}U)(1)=i(U^{-1}\widehat{D}U)(1)=iU^{-1}\widehat{D}(1)=iU^{-1}(0)=0.
\]
This implies that 
\begin{multline*}
0=(U^{-1}i\widehat{D}U)(1)=(M_{1 \otimes h}+i\widehat{D'})(1)\\
=1 \otimes h+i\widehat{D'}(1)=1\otimes h +i0=1\otimes h.
\end{multline*}
Thus we have $U^{-1}i\widehat{D}U=i\widehat{D'}$. This implies that for any $f \in \Lip(X_1)$, we have 
$(U^{-1}i\widehat{D}U)(f\otimes 1)(x)=i\widehat{D'}(f \otimes 1)(x)=0$. In addition by the definition of $D$, we get 
\begin{equation}\label{22}
\begin{split}
(U^{-1}i\widehat{D}U)(f\otimes 1)&=U^{-1}(i \widehat{D}U(f \otimes1))\\
&=iU^{-1}(1 \otimes b U(f \otimes 1)-U(f \otimes 1)1 \otimes b).
\end{split}
\end{equation}
Therefore we  have
\[
U^{-1}(1 \otimes bU(f \otimes 1)-U(f \otimes 1)1 \otimes b)=0.
\]
Since $U$ is surjective, we have
\begin{equation}\label{33}
1 \otimes bU(f \otimes 1)=U(f \otimes 1)1 \otimes b.
\end{equation}
Note we choose $b \in \mathcal{A}_2$ with $b^{*}=-b$ arbitrary. For each $a \in \mathcal{A}_2$ there exist unique elements $b_1, b_2 \in \ma_2$ such that $b_i^{*}=-b_i$ for $i=1,2$ and $a=-ib_1+b_2$. By applying (\ref{33}), we have 
\[
aU(f \otimes 1)(y)=U(f \otimes1)(y)a
\]
for any $a \in \ma_2$ and  $y \in X_2$. We get $U(f \otimes 1)(y) \in \mathcal{A} \cap \ma'=\mathbb{C}1$. Thus there is $g(y) \in \mathbb{C}$ such that
$U(f \otimes1)(y)=g(y)1$.
Since $U(f \otimes1) \in \lxtm$, we get $g \in \Lip(X_2)$ and
\[
U(f\otimes 1)=g\otimes 1.
\]
Thus we can define a map $P_U: \Lip(X_1) \to \Lip(X_2)$ by
\[
U(f\otimes 1)=P_U(f) \otimes 1, \quad f \in \Lip(X_1).
\]
It is easy to see that $P_U$ is a surjective complex linear isometry. Applying \cite[Corollary 15]{HO1}, there is a surjective isometry $\varphi: X_2 \to X_1$ such that
\[
U(f \otimes 1)(y)=P_U(f)(y)\otimes 1=f(\varphi(y)) \otimes 1, \quad f \in \Lip(X_1),\,\, y \in X_2.
 \]
\end{proof}

\begin{proof}[Proof of Theorem \ref{isometrydd}]
A simple calculation shows that the map $U$ from $\lxom$ onto $\lxtm$, which has the form of the theorem is a unital surjective linear isometry. We show the converse.
For any $h \in \hmo$, there exists $\psi_0(h) \in \hmt$ and $*$-derivation $D$ on $\ma_2$ such that 
\[
UM_{1 \otimes h}U^{-1}=M_{1 \otimes \psi_0(h)}+i \widehat{D}.
\]
Let $f \in \Lip(X_1)$. By Lemma \ref{4}, we have 
\begin{equation*}
\begin{split}
U(f \otimes h)(y)&=U(M_{1 \otimes h}(f \otimes 1))(y)=UM_{1 \otimes h}U^{-1}U(f \otimes 1)(y)\\
&=(M_{1 \otimes \psi_0(h)}+i \widehat{D})(U(f \otimes 1))(y)\\
&=\psi_0(h)(U(f \otimes 1)(y))=f(\varphi(y))\psi_0(h)
\end{split}
\end{equation*}
for any $y \in X_2$. For any $a \in \ma_1$, there exist $h_1,h_2 \in \hmo$ such that $a=h_1+ih_2$ and we get
\begin{equation*}
\begin{split}
U(f \otimes a)(x)&=U(f \otimes (h_1+ih_2))(x)
=U(f \otimes h_1)(x)+iU(f \otimes h_2)(x)\\
&=f(\varphi(x))\psi_0(h_1)+if(\varphi(x))\psi_0(h_2)\\
&=f(\varphi(x))\psi(a)
=\psi((f \otimes a)(\varphi(x)))=\psi(f(\varphi(y)) a)
\end{split}
\end{equation*}
for any $f \in \Lip(X_1)$ and $a \in \mathcal{A}_1$. Applying Proposition \ref{3.6} we obtain
 \[
UF(y)=\psi(F(\varphi(y))), \quad  F \in \lxom, y \in X_2.
\] 
\end{proof}

\end{section}

\begin{section}{Concluding comments and remarks}
Let us look at further problems related to Theorem \ref{isometrydd}. It is natural to investigate the following questions; What is the general form of unital surjective linear isometries between $\Lip(X,\ma)$, where $\ma$ is a unital $C^{*}$-algebra?  What is a complete description of surjective linear isometries on $\Lip(X,\ma)$ without the assumption that isometries preserve the identity?  In fact, less is known about surjective linear isometries on Banach spaces of all vector-valued Lipschitz maps with $\|\cdot\|_{L}$. The author suspects the reason  relies on a lack of a complete characterization of  the  extreme points of $\mathbb{B}((\lxe)^{*})$.  Thus we believe Theorem \ref{hermitian} is one of  crucial tools in investigating 
our questions. This might be an interesting direction for further research. These are left as research problems in the future.

\subsection*{Acknowledgments}
This work was supported by JSPS KAKENHI Grant Numbers JP21K13804. 

\end{section}


\begin{thebibliography}{99}

\bibitem{BD}
F.~F.~Bonsall and J.~Duncan,
\emph{Numerical Ranges of Operators on Normed Spaces and of Elements of Normed Algebras,} 
London Mathematical Society Lecture Note Series, vol. 2, Cambridge University Press, London–New York (1971).





\bibitem{BJJMm}
F.~Botelho, J.~Jamison, A.~Jim\'{e}nez-Vargas and M.~Villegas-Vallecillos,
\emph{Hermitian operators on Lipschitz function spaces,} Studia Math., {\bf 215} (2013), 127--137.



\bibitem{BJJMs}
F.~Botelho, J.~Jamison, A.~Jim\'{e}nez-Vargas and M.~Villegas-Vallecillos,
\emph{Hermitian operators on Banach algebras of Lipschitz functions,} Proc. Amer. Math, Soc. {\bf142} (2014), 3469--3481.


\bibitem{FJ}
R.~J.~Fleming and J.~E.~Jamison,
\emph{Hermitian Operators on $C(X,E)$ and the Banach-Stone Theorem,} Math. Z., {\bf 170} (1980), 77--84.

\bibitem{FJ89}
R.~J.~Fleming and J.~E.~Jamison,
\emph{Hermitian operators and isometries on sums of Banach spaces,} Proc. Edinburgh Math. Soc. {\bf  32}, (1989), 169--191.

\bibitem{FJB}
R.~J.~Fleming and J.~E.~Jamison,
\emph{Isometries on Banach spaces,} 
Chapman \& Hall/CRC Monographs and Surveys in Pure and Applied Mathematics, 129. Chapman \& Hall/CRC, Boca Raton, FL, 2003.

\bibitem{FJB08}
R.~J.~Fleming and J.~E.~Jamison,
\emph{Isometries on Banach spaces Vol. 2. Vector-Valued Function Spaces,} 
Chapman \& Hall/CRC Monogr. Surveys Pure Appl. Math. 138, Chapman \& Hall/CRC, Boca Raton, FL, 2008.



\bibitem{kho}
O.~Hatori, K.~Kawamura and S.~Oi, 
\emph{Hermitian operators and isometries on injective tensor products of uniform algebras and $C^{*}$-algebras,} J. Math. Anal. Appl. 
{\bf 472} (2019), 827--841.

\bibitem{HO1}
O.~Hatori and S.~Oi, 
\emph{Isometries on Banach algebras of vector-valued maps,} Acta Sci. Math. (Szeged),  {\bf 84} (2018), 151--183.



\bibitem{kadison}
R.~V.~Kadison,
\emph{Isometries of operator algebras,} Ann.of Math., {\bf 54} (1951), 325--338.

\bibitem{Lu}
G.~Lumer, 
\emph{Semi-inner product of bounded maps into Banach space,} Trans. Amer. Math. Soc. {\bf 100} (1961), 26--43.

\bibitem{Lu2}
G.~Lumer, 
\emph{On the isometries of reflexive Orlicz spaces,} Ann. Inst. Fourier, (1963), 99--109.

\bibitem{My}
S.~B.~ Myers,
\emph{
Banach spaces of continuous functions,}
Ann. of Math., {\bf 49} (1948), 132--140.


\bibitem{hermitian1}
S.~Oi, 
\emph{Hermitian operators and isometries on algebras of matrix-valued Lipschitz maps,} Linear and Multilinear Algebra,  {\bf 68} (2020), 1096--1112.




\bibitem{Si}
A.~M.~Sinclair, 
\emph{Homomorphisms and Derivations on Semisimple Banach Algebras,} Proc. Amer. Math, Soc. {\bf 24} (1970), 209--214.

  

\end{thebibliography}
\end{document}